\documentclass{amsart}
\newtheorem{lemma}{Lemma}
\newcommand{\ef}{\mathbb{F}}
\newcommand{\er}{\mathbb{R}}

\newcommand{\zet}{\mathbb{Z}}
\newcommand{\ce}{\mathbb{C}}

\newtheorem{theorem}{Theorem}
\newtheorem{corollary}{Corollary}
\newtheorem{conjecture}{Conjecture}
\numberwithin{equation}{section}
\begin{document}
\title[On the Erd\H{o}s-Falconer distance problem]
{On the Erd\H{o}s-Falconer distance problem
for two sets of different size in vector spaces over finite fields}
\author{Rainer Dietmann}
\address{Department of Mathematics, Royal Holloway, University of London\\
TW20 0EX Egham, United Kingdom}
\email{Rainer.Dietmann@rhul.ac.uk}
\maketitle
\begin{abstract}
We consider a finite fields version of the Erd\H{o}s-Falconer distance
problem for two different sets. In a certain range for the sizes of the
two sets we obtain results of the conjectured order of magnitude.
\end{abstract}

\section{Introduction}
Let $E \subset \er^s$, and let
\[
  \Delta(E) = \{\|\mathbf{x}-\mathbf{y}\| : \mathbf{x}, \mathbf{y} \in E\}
\]
be the set of distances between elements in $E$, where $\|\cdot\|$ denotes
the Euclidean metric. Erd\H{o}s' distance conjecture \cite{E} is that
\begin{equation}
\label{august}
  \#\Delta(E) \gg_\epsilon \; (\#E)^{s/2-\epsilon}
\end{equation}
for $s \ge 2$ and finite $E$. In a recent breakthrough paper by Guth and
Katz \cite{GK}, this problem has been solved for $s=2$, whereas it is still
open for higher dimensions.
Later Falconer \cite {F} considered a
continuous version of Erd\H{o}s' distance problem, replacing $\#E$ by the
Hausdorff dimension of $E$, and $\#\Delta(E)$ by the Lebesgue measure of
$\Delta(E)$. More recently, Iosevich and Rudnev \cite{IR} 
dealt with a finite fields version of these problems.
For a finite field $\ef_q$ and $\mathbf{x} \in \ef_q^s$, let
\[
  |\mathbf{x}|^2 = \sum_{i=1}^s x_i^2.
\]
Note that this is a natural way of defining distance over finite fields,
as for Euclidean distance keeping the property of being invariant under
orthogonal transformations, whereas on the other hand $|\mathbf{x}|^2=0$
no longer implies that $\mathbf{x}=\mathbf{0}$, since for $s \ge 3$ all
quadratic forms over finite fields are isotropic.\\
In the following we will always assume that $q$ is odd; in particular,
$q \ge 3$.
As pointed out in the introduction of \cite{IR},
the conjecture \eqref{august} no longer holds true over finite fields
irrespective of the size of $E$. One example (see introduction of
\cite{CEHIK}) for this phenomenon are sets
$E$ small enough to fall prey to certain number theoretic properties
of $\ef_q$:
Let $q$ be a prime such that $q \equiv 1 \pmod 4$, and let $i \in \ef_q$ 
be a square root of $-1$. For the set
\[
  E=\{(x,ix) : x \in \ef_q\}
\]
in $\ef_q^2$
one then immediately verifies that $\# E=q$, but $\# \Delta(E) = 1$.
For sets of large enough size, however, one should expect $\Delta(E)$ to have
order of magnitude $q$ many elements,
or even be the set of all elements in $\ef_q$.
In this context, one of Iosevich and Rudnev's main results
(see \cite{IR}, Theorem 1.2) is that if
$E \subset \ef_q^s$ where $\#E \ge Cq^{s/2}$ for a sufficiently large
constant $C$, then
\begin{equation}
\label{goldener_reiter}
  \#\Delta(E) \gg \min\left\{q, \frac{\#E}{q^{(s-1)/2}} \right\},
\end{equation}
where
\[
  \Delta(E) = \left\{|\mathbf{x}-\mathbf{y}|^2 : \mathbf{x}, \mathbf{y}
  \in E\right\}.
\]
In particular, if $\#E \gg q^{(s+1)/2}$, then $\#\Delta(E)\gg q$.
For $s=2$, the stronger result that $\#\Delta(E) \gg q$ if
\begin{equation}
\label{depris}
  \#E \gg q^{4/3}
\end{equation}
has recently been established by Chapman, Erdogan, Hart, Iosevich and Koh
(see \cite{CEHIK}, Theorem 2.2).
Our focus in this paper is on a generalisation of this
problem to the situation of distances between two different sets
$E, F \in \ef_q^s$. Analogously to above, we define
\[
  \Delta(E, F)=\#\{|\mathbf{x}-\mathbf{y}|^2 : \mathbf{x} \in E,
  \, \mathbf{y} \in F\}.
\]
It is straightforward to adapt Iosevich and Rudnev's approach to show that
if $(\#E)(\#F) \ge Cq^s$ for a sufficiently large constant $C$, then
\begin{equation}
\label{schnee}
  \#\Delta(E,F) \gg \min\left\{q, \frac{(\#E)^{1/2}(\#F)^{1/2}}
  {q^{(s-1)/2}}\right\};
\end{equation}
see also Theorem 2.1 in \cite{S} for a similar result.
In particular, if $(\#E)(\#F) \gg q^{s+1}$, then $\#\Delta(E,F) \gg q$.
For $s=2$, the stronger result that $\#\Delta(E,F) \gg q$ if
\begin{equation}
\label{delta}
  (\#E)(\#F) \gg q^{8/3}
\end{equation}
has recently been proved by Koh and Shen
(\cite {KS}, Theorem 1.3), this way generalising \eqref{depris},
and they also put forward the following
conjecture (see Conjecture 1.2 in \cite{KS2})
generalising Conjecture 1.1 in \cite{IR} for even $s$.
\begin{conjecture}
\label{dienstag}
Let $s \ge 2$ be even and $(\#E)(\#F) \ge Cq^s$ for a sufficiently large
constant $C$. Then $\#\Delta(E,F) \gg q$.
\end{conjecture}
In this paper we establish the following result, which improves on
(\ref{schnee}) and (\ref{delta}) for sets $E, F$ of
different size in a certain range for $(\#E)$ and $(\#F)$.
\begin{theorem}
\label{chaos}
Let $E, F \subset \ef_q^s$ where $s \ge 2$. Further, let $\#E \le
\#F$ and $(\#E)(\#F) \ge (900+\log q)q^s$. Then
\begin{equation}
\label{sonntag}
  \#\Delta(E,F) \gg \min \left\{q, \frac{\#F}{q^{(s-1)/2}\log q} \right\}.
\end{equation}
For $s=2$ also the alternative lower bound
\begin{equation}
\label{montag}
  \#\Delta(E,F) \gg \min \left\{q, \frac{(\#E)^{1/2}{\#F}}{q\log q} \right\}
\end{equation}
holds true.
\end{theorem}
Note that (\ref{montag}) is superior to (\ref{sonntag}) for $s=2$ if and only
if $\#E \gg q$.
Note also that Theorem \ref{chaos}
implies that if $(\#E)(\#F) \ge (900+\log q)q^s$ and
$\max\{\#E, \#F\} \ge
q^{(s+1)/2} \log q$, then $\#\Delta(E,F) \gg q$.
These conditions on $E$ and $F$
are for example satisfied if $\#E \ge q^{(s-1)/2}$ and
$\#F \ge (900+\log q) q^{(s+1)/2}$.
Hence, apart from a factor $\log q$, Conjecture \ref{dienstag}
holds true for a certain range of cardinalities
of $E$ and $F$, both for even and odd dimension $s$.\\
Our approach follows that of Iosevich and Rudnev, paying close
attention to certain spherical averages of Fourier transforms.

\section{Notation}
Our notation is fairly standard. Let $\ce$ be the field of complex
numbers, and we write $\ef_q$ for a fixed finite field
having $q$ elements, where $q$ is odd, and we denote by
$\ef_q^*$ the non-zero elements of $\ef_q$. Further, if $a \in \ef_q^*$,
we write $\overline{a}$ for the multiplicative inverse of $a$. Moreover,
we write
\[ 
  e\left(\frac{j}{q} \right) \; (1 \le j \le q)
\]
for the additive
characters of $\ef_q$, the main character being that where $j=q$.
If $q$ is a prime, then $e(j/q)$ is just
\[
  e\left(\frac{j}{q}\right) = e^{2\pi i \frac{j}{q}}
\]
where $i^2=-1$. If $f : \ef_q^s \rightarrow \ce$ is any function, then
we denote by $\hat{f}$ its Fourier transform given by
\[
  \hat{f}(\mathbf{x}) = q^{-s} \sum_{\mathbf{m} \in \ef_q^s}
  e\left(\frac{-\mathbf{m}\mathbf{x}}{q}\right) f(\mathbf{m}),
\]
where as usual $\mathbf{m}\mathbf{x}$ is the inner product
\[
  \mathbf{m}\mathbf{x}=\sum_{i=1}^s m_ix_i.
\]
The function $f$ can be recovered from its Fourier transform $\hat{f}$
via the inversion formula
\[
  f(\mathbf{x}) = \sum_{\mathbf{m} \in \ef_q^s}
  e\left(\frac{\mathbf{m}\mathbf{x}}{q}\right)\hat{f}(\mathbf{m}).
\]
The tool that underpins many arguments is \textit{Plancherel's
formula}
\[
  \sum_{\mathbf{m} \in \ef_q^s} \left|\hat{f}(\mathbf{m})\right|^2
  = q^{-s} \sum_{\mathbf{x} \in \ef_q^s} \left|f(\mathbf{x})\right|^2.
\]
All these formulas are easy to verify, and proofs can be found in many
textbooks on number theory or Fourier analysis. For a subset
$E \subset \ef_q^s$, we also write $E$ for its characteristic function,
i.e.
\[
  E(\mathbf{x}) = 
  \begin{cases}
    1 & \text{if $\mathbf{x} \in E$,} \\
    0 & \text{otherwise,}
  \end{cases}
\]
and analogously for subsets $F \subset \ef_q^s$. Moreover, let $S_r$ be
the sphere
\[
  S_r = \{\mathbf{x} \in \ef_q^s:|\mathbf{x}|^2=r\},
\]
and as above we also write $S_r$ for the corresponding characteristic
function.
Moreover, for $E \subset \ef_q^s$ and $r \in \ef_q$, let $\sigma_E(r)$ be the
spherical average
\[
  \sigma_E(r) = \sum_{\mathbf{a} \in \ef_q^s:
  |\mathbf{a}|^2 = r} |\hat{E}(\mathbf{a})|^2
\]
of the Fourier transform $\hat{E}(\mathbf{a})$ of $E$, and we
define analogously $\sigma_F(r)$. Furthermore, we define
\[
  \sigma_{E, F}(r) = \sum_{\mathbf{m} \in \ef_q^s:|\mathbf{m}|^2=r}
  \overline{\hat{E}(\mathbf{m})} \hat{F}(\mathbf{m}),
\]
where as usual $\, \bar{} \, $ denotes complex conjugation.
In particular, $\sigma_E(r)=\sigma_{E,E}(r)$.
Our main tool for bounding $\#\Delta(E,F)$ below is the following
upper bound on $\sigma_E \sigma_F$ on average. In the following,
all implied $O$-constants depend at most on the dimension $s$.
\begin{lemma}
\label{tavanic}
In the notation from above, let $s \ge 2$. Then we have
\begin{equation}
\label{federolf}
  \sum_{r \in \ef_q^*} \sigma_E(r) \sigma_F(r) \ll \log q
  \left( q^{-2s-1} (\#E) (\#F) + q^{-\frac{5s+1}{2}} (\#E)^2 (\#F) \right).
\end{equation}
For odd $s \ge 2$, also the bound
\begin{equation}
\label{kiste}
  \sum_{r \in \ef_q} \sigma_E(r) \sigma_F(r) \ll \log q
  \left( q^{-2s-1} (\#E) (\#F) + q^{-\frac{5s+1}{2}} (\#E)^2 (\#F) \right).
\end{equation}
holds true, including the term $r=0$. Moreover,
for $s=2$ we also have the alternative bound
\begin{equation}
\label{lidl}
  \sum_{r \in \ef_q^*} \sigma_E(r) \sigma_F(r) \ll
  (\log q) q^{-5} (\#E)^{3/2} (\#F).
\end{equation}
\end{lemma}
Note that (\ref{lidl}) is superior to (\ref{federolf}) for $s=2$ if and
only if $\#E \gg q$. Finally, for fixed $E, F \in \ef_q^s$ and
given $j \in \ef_q$, we define
\begin{equation}
\label{mechanic}
  \nu(j) = \#\{(\mathbf{x}, \mathbf{y}) \in E \times F:
  |\mathbf{x}-\mathbf{y}|^2 =j\}.
\end{equation}

\section{Bounding the Fourier transform of a sphere}
In this section we collect some useful bounds on the Fourier
transform of a sphere in the finite fields setting.
\begin{lemma}
\label{kloostermania}
For $\mathbf{m} \in \ef_q^s$, let
\[
  \chi(\mathbf{m}) =
  \begin{cases}
    1 & \text{if $\mathbf{m}=\mathbf{0}$,} \\
    0 & \text{if $\mathbf{m} \ne \mathbf{0}$}.
  \end{cases}
\]
Then
\[
  \hat{S}_r(\mathbf{m}) = \frac{\chi(\mathbf{m})}{q}
  + q^{-\frac{s}{2}-1} c_q^s
  \sum_{j \in \ef_q^*} e\left( \frac{jr+|\mathbf{m}|^2 \bar{4}
  \bar{j}}{q} \right) \eta_q^s(j),
\]
where the complex number $c_q$ depends only on $q$ and $s$, such that $|c_q|=1$,
and where $\eta_q$ denotes a quadratic multiplicative character of $\ef_q^*$.
\end{lemma}
\begin{proof}
This is Lemma 4 in \cite{IK}.
\end{proof}
\begin{corollary}
\label{zoff}
Let $\mathbf{m} \ne \mathbf{0}$. Then
\[
  |\hat{S}_r(\mathbf{m})| \le q^{-s/2}.
\]
Moreover, still assuming $\mathbf{m} \ne \mathbf{0}$,
for $r \ne 0$ or odd $s$, the stronger bound
\[
  \hat{S}_r(\mathbf{m}) \ll q^{-\frac{s+1}{2}}
\]
holds true. Further, for $s \ge 2$ and $\mathbf{m}=\mathbf{0}$ we have
the bound
\[
  |\hat{S}_r(\mathbf{0})| \le \frac{2}{q}.
\]
Finally,
\[
  \hat{S}_0(\mathbf{m}) = c_q^s(q^{-s/2}-q^{-s/2-1})
\]
for $\mathbf{m} \ne \mathbf{0}$, $|\mathbf{m}|^2=0$ and even $s$, and
\[
  \hat{S}_0(\mathbf{m}) \ll q^{-s/2-1}
\]
for $\mathbf{m} \ne \mathbf{0}$, $|\mathbf{m}|^2 \ne 0$ and even $s$.
\end{corollary}
\begin{proof}
The first and third bound follow immediately from Lemma 
\ref{kloostermania} on trivially bounding the sum over $j$.
For the second one we make use of Weil's seminal work (see for example
Corollary 11.12 in \cite{Bibel})
to bound the resulting Kloosterman sum over $j$ (even $s$), or use the
elementary evaluation of the Sali\'{e} sum
(see for example Lemma 12.4 in \cite{Bibel})
to bound the relevant sum over $j$ (odd $s$).
The last two bounds follow on evaluating the summation over $j$
after noting that the term $\eta_q^s(j)$ vanishes for even $s$.
\end{proof}
\begin{lemma}
\label{karte}
Let $s \ge 2$ and $r \in \ef_q^*$. Then
\begin{equation}
\label{wirklich}
  \sigma_E(r) \ll q^{-s-1} \# E + q^{-\frac{3s+1}{2}} (\#E)^2.
\end{equation}
This bound is also true for $r=0$ and odd $s$. Moreover,
for $s=2$, we also have the alternative bound
\begin{equation}
\label{zum_kotzen}
  \sigma_E(r) \ll q^{-3} (\#E)^{3/2}.
\end{equation}
\end{lemma}
\begin{proof}
The bound \eqref{wirklich} for $r \ne 0$ is essentially Lemma 1.8 in
\cite{IR}, but in order to cover the case $r=0$ and odd $s$ as well let
us give a complete proof. We have
\begin{align*}
  \sigma_E(r) & = \sum_{\mathbf{m} \in \ef_q^s:|\mathbf{m}|^2=r}
  |\hat{E}(\mathbf{m})|^2
  = \sum_{\mathbf{m} \in \ef_q}
  \overline{\hat{E}(\mathbf{m})} \hat{E}(\mathbf{m}) S_r(\mathbf{m}) \\
  & = q^{-2s} \sum_{\mathbf{m} \in \ef_q^s}
  \sum_{\mathbf{x} \in \ef_q^s} E(\mathbf{x})
  e\left(\frac{\mathbf{m}\mathbf{x}}{q}\right)
  \sum_{\mathbf{y} \in \ef_q^s} E(\mathbf{y})
  e\left(\frac{-\mathbf{m}\mathbf{y}}{q}\right)
  S_r(\mathbf{m}) \\
  & = q^{-2s} \sum_{\mathbf{x}, \mathbf{y} \in \ef_q^s}
  E(\mathbf{x}) E(\mathbf{y}) \sum_{\mathbf{m} \in \ef_q^s}
  e\left(\frac{\mathbf{m}(\mathbf{x}-\mathbf{y})}{q}\right)
  S_r(\mathbf{m}) \\
  & = q^{-s} \sum_{\mathbf{x}, \mathbf{y} \in \ef_q^s}
  E(\mathbf{x}) E(\mathbf{y}) \hat{S}_r(\mathbf{y}-\mathbf{x}) \\
  & \le q^{-s} \# E \cdot |\hat{S}_r(\mathbf{0}) |+
  q^{-s} (\#E)^2 \max_{\mathbf{m} \in \ef_q^s \backslash\{\mathbf{0}\}}
  |\hat{S}_r(\mathbf{m})|.
\end{align*}
Corollary \ref{zoff} now yields \eqref{wirklich}. The second bound
(\ref{zum_kotzen}) is Lemma 4.4 in \cite{CEHIK}.
\end{proof}

\section{Proof of Lemma \ref{tavanic}}
Clearly, by Plancherel's formula,
\[
  \sigma_F(r) \le \sum_{\mathbf{a} \in \ef_q^s}
  |\hat{F}(\mathbf{a})|^2 = \frac{|F|}{q^s} \le 1,
\]
and the same bound holds true for $\sigma_E(r)$.
Hence, on writing
\[
  T_i = \sum_{r \in \ef_q^*:2^{i-1} \le \sigma_F(r) \le 2^i}
  \sigma_E(r) \sigma_F(r)
\]
for $i \in \zet$,
by a dyadic intersection of the range of possible values of $\sigma_F$
we find that
\[
  \sum_{r \in \ef_q^*} \sigma_E(r) \sigma_F(r) \le q^{-4s+1} +
  \sum_{-4s \frac{\log q}{\log 2} \le i \le 0} T_i
  \ll q^{-4s+1} +
  \log q \cdot \max_{-4s \frac{\log q}{\log 2} \le i \le 0} T_i.
\]
We conclude that there exists a subset $M \subset \ef_q^*$ such that
\begin{equation}
\label{M}
  \sum_{r \in \ef_q^*} \sigma_E(r) \sigma_F(r) \ll q^{-4s+1} + \log q
  \sum_{r \in M} \sigma_E(r) \sigma_F(r)
\end{equation}
and
\begin{equation}
\label{prospectus}
  A \le \sigma_F(r) \le 2A
\end{equation}
for all $r \in M$, for a suitable positive constant $A$.
By Cauchy-Schwarz,
\begin{equation}
\label{N}
  \sum_{r \in M} \sigma_E(r) \sigma_F(r) \le
  \left( \sum_{r \in M} \sigma_E(r)^2 \right)^{1/2}
  \left( \sum_{r \in M} \sigma_F(r)^2 \right)^{1/2}.
\end{equation}
Let us first bound $\sum_{r \in M} \sigma_E(r)^2$. Using Lemma
\ref{karte}, we obtain 
\begin{equation}
\label{L}
  \sum_{r \in M} \sigma_E(r)^2 \le \left(
  \max_{t \in \ef_q^*} \sigma_E(t) \right)^2 \#M \ll (\#M) \left(
  q^{-2s-2} (\#E)^2 + q^{-3s-1} (\#E)^4 \right)
\end{equation}
in general, and for $s=2$ we also obtain the alternative bound
\begin{equation}
\label{alternative}
  \sum_{r \in M} \sigma_E(r)^2 \ll (\#M) q^{-6} (\#E)^3.
\end{equation}
Next, let us  bound $\sum_{r \in M} \sigma_F(r)^2$.
\begin{lemma}
\label{von1717}
We have
\[
  \sum_{r \in \ef_q} \sigma_F(r) = q^{-s} \#F.
\]
\end{lemma}
\begin{proof}
Since
\[
  \sum_{r \in \ef_q} \sigma_F(r) = \sum_{\mathbf{a} \in \ef_q^s}
  |\hat{F}(\mathbf{a})|^2,
\]
the result follows immediately from Plancherel's formula
\[
  \sum_{\mathbf{a} \in \ef_q^s} |\hat{F}(\mathbf{a})|^2
  = q^{-s} \sum_{\mathbf{a} \in \ef_q^s} F(a)^2 = q^{-s} \#F.
\]
\end{proof}
We start with the observation that by (\ref{prospectus}), we have
\begin{equation}
\label{cdu}
  \sum_{r \in M} \sigma_F(r)^2 \le 4 \cdot \#M \cdot A^2.
\end{equation}
Next, by Lemma \ref{von1717},
\begin{equation}
\label{spd}
  q^{-2s} (\#F)^2 = \left( \sum_{r \in \ef_q} \sigma_F(r) \right)^2
  = \sum_{m, n \in \ef_q} \sigma_F(m) \sigma_F(n).
\end{equation}
Moreover, by (\ref{prospectus}),
\begin{equation}
\label{fdp}
   \sum_{m, n \in \ef_q} \sigma_F(m) \sigma_F(n) \ge
   \sum_{m, n \in M} \sigma_F(m) \sigma_F(n) \gg (\#M)^2 A^2.
\end{equation}
By (\ref{cdu}), (\ref{spd}), and (\ref{fdp}) we obtain
\begin{align}
\label{kubus}
  \sum_{r \in M} \sigma_F(r)^2 & \ll \#M \cdot A^2
  \ll (\#M)^{-1}
   \sum_{m, n \in M} \sigma_F(m) \sigma_F(n) \nonumber \\
  & \ll (\#M)^{-1} q^{-2s} (\#F)^2.
\end{align}
Summarising (\ref{M}), (\ref{N}), (\ref{L}) and
(\ref{kubus}), and noting that
\[
  q^{-4s+1} \ll (\log q)q^{-2s-1}(\#E)(\#F)
\]
since $\#E, \#F \ge 1$, we obtain
\[
  \sum_{r \in \ef_q^*} \sigma_E(r) \sigma_F(r) \ll
  (\log q) \left( q^{-2s-1} (\#E) (\#F) +
  q^{-\frac{5s+1}{2}} (\#E)^2 (\#F) \right).
\]
In case of odd $s$, Lemma \ref{karte} also applies for $r=0$, so
in the argument above we can replace $\ef_q^*$ by $\ef_q$, this way
arriving at \eqref{kiste}.
Further,
using (\ref{alternative}) instead of (\ref{L}), for $s=2$ we also obtain
\[
  \sum_{r \in \ef_q^*} \sigma_E(r) \sigma_F(r) \ll
  (\log q) q^{-5} (\#E)^{3/2} (\#F).
\]
This completes the proof of Lemma \ref{tavanic}. \qed

\section{Preparations for the proof of Theorem \ref{chaos}}
Before we embark on the proof of Theorem \ref{chaos}, we first need to
collect some useful lemmata.
\begin{lemma}
\label{insel}
Let $j \in \ef_q$. Then
\[
  \nu(j) = q^{2s} \sum_{\mathbf{m} \in \ef_q^s}
  \hat{S}_j(\mathbf{m}) \overline{\hat{E}(\mathbf{m})} \hat{F}(\mathbf{m}).
\]
\end{lemma}
\begin{proof}
We have
\begin{align*}
  \nu(j) & = \sum_{\mathbf{x}, \mathbf{y} \in \ef_q^s}
  E(\mathbf{x}) F(\mathbf{y}) S_j(\mathbf{x}-\mathbf{y}) \\
  & = \sum_{\mathbf{x}, \mathbf{y} \in \ef_q^s} E(\mathbf{x})
  F(\mathbf{y}) \sum_{\mathbf{m} \in \ef_q^s}
  e\left(\frac{(\mathbf{x}-\mathbf{y})\mathbf{m}}{q}\right)
  \hat{S}_j(\mathbf{m}) \\
  & = \sum_{\mathbf{m} \in \ef_q^s} \hat{S}_j(\mathbf{m})
  \left( \sum_{\mathbf{x} \in \ef_q^s} E(\mathbf{x})
  e\left(\frac{\mathbf{x}\mathbf{m}}{q} \right) \right)
  \left( \sum_{\mathbf{y} \in \ef_q^s} F(\mathbf{y})
  e\left(\frac{-\mathbf{y}\mathbf{m}}{q} \right) \right) \\
  & = q^{2s} \sum_{\mathbf{m} \in \ef_q^s}
  \hat{S}_j(\mathbf{m}) \overline{\hat{E}(\mathbf{m})} \hat{F}(\mathbf{m}).
\end{align*}
\end{proof}

\begin{lemma}
\label{fueller}
Let $s \ge 2$ and $(\#E)(\#F) \ge 900q^s$. Then
\[
  \nu(0) \le \frac{21}{30} (\# E) (\# F).
\]
\end{lemma}
\begin{proof}
Since
\[
  \overline{\hat{E}(\mathbf{0})} = q^{-s} \# E
\]
and
\[
  \hat{F}(\mathbf{0}) = q^{-s} \# F,
\]
Lemma \ref{insel} yields
\[
  \nu(0) = (\#E) (\#F) \hat{S}_0(\mathbf{0}) + \delta,
\]
where
\[
  \delta = q^{2s} \sum_{\mathbf{m} \in \ef_q^s: \mathbf{m} \ne
  \mathbf{0}} \hat{S}_0(\mathbf{m}) \overline{\hat{E}(\mathbf{m})}
  \hat{F}(\mathbf{m}).
\]
By Corollary \ref{zoff}, it follows that
\[
  \nu(0) \le \frac{2(\#E)(\#F)}{q} + |\delta|.
\]
Moreover, Corollary \ref{zoff} gives
\[
  \left|\hat{S}_0(\mathbf{m})\right| \le q^{-s/2}
\]
for $\mathbf{m} \ne \mathbf{0}$. Hence, by Cauchy-Schwarz and
Plancherel's formula,
\begin{align*}
  |\delta| & \le q^{\frac{3}{2}s}
  \left( \sum_{\mathbf{m} \in \ef_q^s} |\hat{E}(\mathbf{m})|^2 \right)^{1/2}
  \left( \sum_{\mathbf{m} \in \ef_q^s} |\hat{F}(\mathbf{m})|^2 \right)^{1/2}
  \\ & \le q^{s/2} (\#E)^{1/2} (\#F)^{1/2}.
\end{align*}
Since $(\#E)(\#F) \ge 900q^s$, we conclude that
\[
  |\delta| \le \frac{(\#E)(\#F)}{30}.
\]
Therefore, since $q \ge 3$, we have
\[
  \nu(0) \le 2 \frac{(\#E)(\#F)}{q} + |\delta| \le
  \frac{21}{30} (\#E)(\#F).
\]
\end{proof}

\begin{lemma}
\label{rollei}
We have
\begin{align*}
  \sum_{j \in \ef_q} \nu(j)^2 & \le
  \frac{(\#E)^2(\#F)^2}{q} + q^{s-1} (\#E)(\#F)\\
  & + q^{3s} \left|\sigma_{E,F}(0)\right|^2
  + q^{3s} \sum_{r \in \ef_q^*} \sigma_E(r) \sigma_F(r)\\
  & \le \frac{(\#E)^2(\#F)^2}{q} +
  q^{3s} \sum_{r \in \ef_q} \sigma_E(r) \sigma_F(r)
  + q^{s-1} (\#E)(\#F). \\
\end{align*}
\end{lemma}
\begin{proof}
By Lemma \ref{insel} and Lemma \ref{kloostermania}, we have
\begin{align*}
  \sum_{j \in \ef_q} \nu(j)^2 & = q^{4s}
  \sum_{j \in \ef_q} \sum_{\mathbf{m}, \mathbf{n} \in \ef_q^s}
  \hat{S}_j(\mathbf{m})
  \overline{S_j(\mathbf{n})}
  \overline{\hat{E}(\mathbf{m})} \hat{F}(\mathbf{m})
  \hat{E}(\mathbf{n}) \overline{\hat{F}(\mathbf{n})} \\
  & = q^{4s} \sum_{\mathbf{m}, \mathbf{n} \in \ef_q^s}
  \overline{\hat{E}(\mathbf{m})} \hat{F}(\mathbf{m})
  \hat{E}(\mathbf{n}) \overline{\hat{F}(\mathbf{n})}
  \sum_{j \in \ef_q} \times \\
  & \times \left( \frac{\chi(\mathbf{m})}{q} + q^{-s/2-1} c_q^s
  \sum_{k \in \ef_q^*} e\left(\frac{kj+|\mathbf{m}|^2 \bar{4}
  \bar{k}}{q}\right) \eta_q^s(k) \right) \\
  & \times \left( \frac{\chi(\mathbf{n})}{q} + q^{-s/2-1} \overline{c}_q^s
  \sum_{l \in \ef_q^*} e\left(\frac{-lj-|\mathbf{n}|^2 \bar{4}
  \bar{l}}{q}\right) \eta_q^s(l) \right). \\
\end{align*}
We are now going to expand the product and interchange the order of
summation of $j$ and $k, l$. Since
\[
  \sum_{l \in \ef_q^*} \sum_{j \in \ef_q}
  e\left(\frac{-lj-|\mathbf{n}|^2 \bar{4} \bar{l}}{q} \right)
  \eta_q^s(l) = 0,
\]
the two cross terms turn out to be zero. Moreover,
\[
  \hat{E}(\mathbf{0}) = \overline{\hat{E}(\mathbf{0})} =
  q^{-s} \# E
\]
and
\[
  \hat{F}(\mathbf{0}) = \overline{\hat{F}(\mathbf{0})} =
  q^{-s} \# F.
\]
Therefore,
\[
  \sum_{j \in \ef_q} \nu(j)^2 =
  \frac{(\#E)^2 (\#F)^2}{q} + q^{3s-2} 
  \sum_{\mathbf{m}, \mathbf{n} \in \ef_q^s}
  \overline{\hat{E}(\mathbf{m})} \hat{F}(\mathbf{m})
  \hat{E}(\mathbf{n}) \overline{\hat{F}(\mathbf{n})}
  T(\mathbf{m}, \mathbf{n}), 
\]
where
\begin{align*}
  T(\mathbf{m}, \mathbf{n}) & =
  c_q^s \overline{c}_q^s \sum_{j \in \ef_q}
  \sum_{k \in \ef_q^*} e\left(\frac{kj+|\mathbf{m}|^2 \bar{4} \bar{k}}{q}
  \right) \eta_q^s(k)
  \sum_{l \in \ef_q^*} e\left(\frac{-lj-|\mathbf{n}|^2 \bar{4} \bar{l}}{q}
  \right) \eta_q^s(l) \\
  & = q \sum_{k \in \ef_q^*}
  e \left( \frac{\bar{4}\bar{k}(|\mathbf{m}|^2-|\mathbf{n}|^2)}{q}
  \right) \\
  & = q \left( \sum_{k \in \ef_q}
  e \left( \frac{\bar{4}k(|\mathbf{m}|^2-|\mathbf{n}|^2)}{q}
  \right) - 1 \right) \\
  & = \left\{ \begin{array}{ll}
    q^2-q & \mbox{ if } \; |\mathbf{m}|^2=|\mathbf{n}|^2 \\
    -q & \mbox{ if } \; |\mathbf{m}|^2 \ne |\mathbf{n}|^2.
  \end{array} \right.
\end{align*}
Hence
\begin{equation}
\label{chertsey}
  \sum_{j \in \ef_q} \nu(j)^2 - \frac{(\#E)^2(\#F)^2}{q} \le U+|V|
\end{equation}
where
\[
  U= q^{3s} \sum_{\mathbf{m}, \mathbf{n} \in \ef_q^s:
  |\mathbf{m}|^2=|\mathbf{n}|^2}
  \overline{\hat{E}(\mathbf{m})} \hat{F}(\mathbf{m})
  \hat{E}(\mathbf{n}) \overline{\hat{F}(\mathbf{n})}
  = q^{3s} \sum_{r \in \ef_q} \left| \sigma_{E,F}(r)\right|^2
\]
and
\[
  V = q^{3s-1} \sum_{\mathbf{m}, \mathbf{n} \in \ef_q^s}
  \overline{\hat{E}(\mathbf{m})} \hat{F}(\mathbf{m})
  \hat{E}(\mathbf{n}) \overline{\hat{F}(\mathbf{n})}.
\]
By Cauchy-Schwarz' inequality,
\[
  |\sigma_{E,F}(r)|^2 \le
  \left( \sum_{\mathbf{m} \in \ef_q^s:|\mathbf{m}|^2=r} |\hat{E}(\mathbf{m})|^2
  \right)
  \left( 
  \sum_{\mathbf{m} \in \ef_q^s:|\mathbf{m}|^2=r} |\hat{F}(\mathbf{m})|^2
  \right)
   = \sigma_E(r) \sigma_F(r).
\]
Thus
\begin{equation}
\label{nhs}
  U \le q^{3s} \left|\sigma_{E,F}(0)\right|^2+
  q^{3s} \sum_{r \in \ef_q^*} \sigma_E(r) \sigma_F(r)
  \le q^{3s} \sum_{r \in \ef_q} \sigma_E(r) \sigma_F(r).
\end{equation}
Another application of Cauchy-Schwarz shows that
\begin{align*}
  \left| \sum_{\mathbf{m}, \mathbf{n} \in \ef_q^s}
  \overline{\hat{E}(\mathbf{m})} \hat{F}(\mathbf{m})
  \hat{E}(\mathbf{n}) \overline{\hat{F}(\mathbf{n})} \right|
  & \le \left(
  \sum_{\mathbf{m} \in \ef_q^s} \left| \hat{E}(\mathbf{m})\right|
  \left|\hat{F}(\mathbf{m})\right| \right)^2 \\
  & \le \sum_{\mathbf{m} \in \ef_q^s} \left| \hat{E}(\mathbf{m})
  \right|^2
  \sum_{\mathbf{m} \in \ef_q^s} \left| \hat{F}(\mathbf{m}) \right|^2.
\end{align*}
Hence, by Plancherel's formula,
\begin{equation}
\label{gatwick}
  |V| \le q^{s-1} (\#E) (\# F).
\end{equation}
The result now follows from (\ref{chertsey}), (\ref{nhs}) and
(\ref{gatwick}).
\end{proof}

\begin{lemma}
\label{montblanc}
Let $s \ge 2$ be even, $\# E \le \# F$ and $(\# E)(\# F) \ge 900q^s$.
Then we have
\[
  \left|\sigma_{E, F}(0)\right|^2 = q^{-3s} \nu(0)^2
  + O\left(q^{-3s-1}(\#E)^2(\# F)^2\right).
\]
\end{lemma}
\begin{proof}
As in the proof of Lemma \ref{karte},
\begin{align*}
  \sigma_{E,F}(0) & = \sum_{\mathbf{m} \in \ef_q^s:
  |\mathbf{m}|^2=0} \overline{\hat{E}(\mathbf{m})} \hat{F}(\mathbf{m})
  = \sum_{\mathbf{m} \in \ef_q^s}
  \overline{\hat{E}(\mathbf{m})} \hat{F}(\mathbf{m}) S_0(\mathbf{m}) \\
  & = q^{-2s}
  \sum_{\mathbf{m} \in \ef_q^s} \sum_{\mathbf{x} \in \ef_q^s}
  E(\mathbf{x}) e\left(\frac{\mathbf{m}\mathbf{x}}{q}\right)
  \sum_{\mathbf{y} \in \ef_q^s}
  F(\mathbf{y}) e\left(\frac{-\mathbf{m}\mathbf{y}}{q}\right)
  S_0(\mathbf{m})\\
  & = q^{-2s} \sum_{\mathbf{x}, \mathbf{y} \in \ef_q^s}
  E(\mathbf{x}) F(\mathbf{y}) \sum_{\mathbf{m} \in \ef_q^s}
  e\left(\frac{\mathbf{m}(\mathbf{x}-\mathbf{y})}{q}\right)
  S_0(\mathbf{m})\\
  & = q^{-s} \sum_{\mathbf{x}, \mathbf{y} \in \ef_q^s}
  E(\mathbf{x}) F(\mathbf{y}) \hat{S}_0(\mathbf{y}-\mathbf{x}).
\end{align*}
By Corollary \ref{zoff} and Cauchy-Schwarz' inequality we obtain
\begin{align*}
  \sigma_{E,F}(0) & = q^{-s} c_q^s \sum_{\mathbf{x}, \mathbf{y} \in
  \ef_q^s : \mathbf{x} \ne \mathbf{y}, |\mathbf{x}-\mathbf{y}|^2=0}
  E(\mathbf{x}) F(\mathbf{y}) \left(q^{-s/2}-q^{-s/2-1}\right)\\
  & + O\left(q^{-s} \sum_{\mathbf{x} \in \ef_q^s} E(\mathbf{x})
  F(\mathbf{x}) q^{-1}\right)\\
  & + O\left( q^{-s} \sum_{\mathbf{x}, \mathbf{y} \in \ef_q^s:
  \mathbf{x} \ne \mathbf{y}, |\mathbf{x}-\mathbf{y}|^2 \ne 0}
  E(\mathbf{x}) F(\mathbf{y}) q^{-s/2-1}\right)\\
  & = q^{-\frac{3}{2}s} c_q^s
  \left( \nu(0) + O(\#E) \right) +
  O\left(q^{-s-1} \sum_{\mathbf{x} \in \ef_q^s}
  E(\mathbf{x}) F(\mathbf{x})\right) \\
  & + O \left( q^{-\frac{3}{2}s-1}
  \sum_{\mathbf{x}, \mathbf{y} \in \ef_q^s: \mathbf{x} \ne \mathbf{y}}
   E(\mathbf{x}) F(\mathbf{y}) \right)\\
  & = q^{-\frac{3}{2}s} c_q^s \nu(0) + O\left(q^{-\frac{3}{2}s} \# E\right)
  \\ & + O\left(q^{-s-1} \left( \sum_{\mathbf{x} \in \ef_q^s}
  E(\mathbf{x})^2 \right)^{1/2}
  \left( \sum_{\mathbf{x} \in \ef_q^s} F(\mathbf{x})^2 \right)^{1/2}
  \right)\\
  & + O\left(q^{-\frac{3}{2}s-1} \sum_{\mathbf{x} \in \ef_q^s}
  E(\mathbf{x}) \sum_{\mathbf{y} \in \ef_q^s} F(\mathbf{y})\right)\\
  & = q^{-\frac{3}{2}s} c_q^s \nu(0)
  + O\left(q^{-\frac{3}{2}s}\# E\right)
  + O\left(q^{-s-1} (\#E)^{1/2} (\#F)^{1/2}\right)\\
  & + O\left(q^{-\frac{3}{2}s-1}(\#E)(\#F)\right)\\
  & = q^{-\frac{3}{2}s} c_q^s \nu(0) + O\left(q^{-\frac{3}{2}s-1}
  (\# E)(\# F)\right).
\end{align*}
Multiplying with $\overline{\sigma_{E,F}(0)}$ and noting that
$\nu(0)=O\left((\# E)(\# F)\right)$ by Lemma \ref{fueller}
then yields the result.
\end{proof}

\begin{lemma}
\label{tourist}
Let $s \ge 2$, $\# E \le \# F$ and $(\#E)(\#F) \ge (\log q+900) q^s$.
Then
\begin{equation}
\label{autist}
  \sum_{r \in \ef_q^*} \nu(r)^2 
  \ll \frac{(\#E)^2(\#F)^2}{q} + (\log q)q^{\frac{s-1}{2}}(\#E)^2(\#F).
\end{equation}
For $s=2$, we also have the alternative bound
\[
  \sum_{r \in \ef_q^*} \nu(r)^2 \ll
  \frac{(\#E)^2(\#F)^2}{q} + (\log q) q (\# E)^{3/2} (\#F).
\]
\end{lemma}
\begin{proof}
By Lemma \ref{rollei} and Lemma \ref{tavanic}, for odd $s \ge 2$ we obtain
\begin{align*}
  \sum_{r \in \ef_q} \nu(r)^2 & \ll
  \frac{(\#E)^2(\#F)^2}{q} + q^{s-1} (\#E)(\#F) \\
  & + (\log q)
  (q^{s-1} (\#E)(\#F)+q^{\frac{s-1}{2}} (\#E)^2 (\#F)).
\end{align*}
Note that
\[
  (\log q)q^{s-1} (\#E)(\#F) \ll \frac{(\#E)^2 (\#F)^2}{q}
\]
since $(\#E)(\#F) \gg (\log q)q^s$, whence
\[
  \sum_{r \in \ef_q} \nu(r)^2 \ll
  \frac{(\#E)^2(\#F)^2}{q} + (\log q) q^{\frac{s-1}{2}} (\#E)^2 (\#F).
\]
Since $\nu(0)^2 \ge 0$, this is even stronger than the claim
\eqref{autist}.
For even $s \ge 2$, Lemma \ref{rollei} and Lemma \ref{montblanc} yield
\begin{align*}
  \sum_{r \in \ef_q} \nu(r)^2 & \le
  \frac{(\#E)^2(\#F)^2}{q} + q^{s-1} (\#E)(\#F) \\
  & + \nu(0)^2 + O(q^{-1} (\#E)^2 (\#F)^2) +
  q^{3s} \sum_{r \in \ef_q^*} \sigma_E(r) \sigma_F(r).
\end{align*}
As above, subtracting $\nu(0)^2$ and applying Lemma \ref{tavanic}
then gives
\[
  \sum_{r \in \ef_q^*} \nu(r)^2 \ll
  \frac{(\#E)^2 (\#F)^2}{q} + (\log q) q^{\frac{s-1}{2}} (\#E)^2 (\#F).
\]
To obtain the alternative
bound for $s=2$, we just use the alternative bound in Lemma \ref{tavanic}
and keep the rest of the proof the same.
\end{proof}

\section{Proof of Theorem \ref{chaos}}
We follow the argument leading to formula (2.6) in \cite{IR}.
By definition (\ref{mechanic}) of $\nu(j)$, clearly
\[
  \sum_{j \in \ef_q} \nu(j) = (\#E)(\#F).
\]
Hence, by Lemma \ref{fueller},
\[
  \left( \sum_{j \in \ef_q} \nu(j) \right)^2 - 2\nu(0)^2
  \ge \frac{1}{50} (\#E)^2(\#F)^2.
\]
Moreover, by Cauchy-Schwarz,
\begin{align*}
  \left( \sum_{j \in \ef_q} \nu(j) \right)^2 & \le 2\nu(0)^2
  + 2 \left( \sum_{j \in \ef_q^*} \nu(j) \right)^2\\
  & \le 2 \nu(0)^2 + 2 \left( \sum_{j \in \ef_q^*} \nu(j)^2 \right)
  \cdot \left( \sum_{j \in \ef_q^*:\nu(j)>0} 1 \right)\\
  & \le 2\nu(0)^2 + 2\#\Delta(E, F) \cdot \sum_{j \in \ef_q^*} \nu(j)^2.
\end{align*}
Thus
\[
  \#\Delta(E, F) \gg \frac{(\#E)^2(\#F)^2}{\sum_{j \in \ef_q^*} \nu(j)^2}.
\]
The conclusion now follows immediately from Lemma \ref{tourist}.

\end{document}